\documentclass[12pt]{article}

\usepackage{graphics,amsmath,amssymb,amsthm,mathrsfs}

\newcommand{\br}{\mathbb{R}}

\newcommand{\varep}{\varepsilon}

\newtheorem{thm}{Theorem}[section]
\newtheorem{lemma}[thm]{Lemma}

\newtheorem{remark}[thm]{Remark}

\numberwithin{equation}{section}

\begin{document}

\title{On $L^p$ Estimates in Homogenization \\ of Elliptic Equations of Maxwell's Type}

\author{Zhongwei Shen
\and Liang Song}

\date{ }

\maketitle

\begin{abstract}

For a family of second-order elliptic systems of Maxwell's type with rapidly oscillating periodic coefficients in a $C^{1, \alpha}$ domain $\Omega$,
we establish uniform estimates of solutions $u_\varep$ and $\nabla \times u_\varep$ in $L^p(\Omega)$
for $1<p\le \infty$. The proof relies on the uniform $W^{1,p}$ and Lipschitz estimates
for solutions of scalar elliptic equations with periodic coefficients.

\end{abstract}


\section{Introduction}

Let $\Omega$ be a  bounded $C^{1, \alpha}$ domain in $\br^3$ for some $\alpha>0$ and $n$ denote the outward unit normal to $\partial\Omega$.
Let $A(y)=(a_{ij}(y))$ and  $B(y)=(b_{ij} (y))$ be  two $3\times 3$ matrices with real entries satisfying the ellipticity
conditions:
\begin{equation}\label{ellipticity}
 \mu |\xi|^2 \le a_{ij} (y) \xi_i\xi_j \le \frac{1}{\mu} |\xi|^2, \quad  \mu |\xi|^2 \le b_{ij} (y) \xi_i\xi_j \le \frac{1}{\mu} |\xi|^2
\end{equation}
for any $\xi, y\in \br^3$ and some $\mu>0$.
Consider the second-order elliptic system of Maxwell's type:
\begin{equation}\label{max-equation}
\nabla \times \big(A(x/\varep) \nabla \times u_\varep\big) + B(x/\varep)u_\varep  = F +\nabla \times G  \quad  \text{ in } \Omega,
\end{equation}
where  $u_\varep$ is a vector field in $\Omega$ and $\varep>0$  a small parameter.
Given $F, G\in L^2(\Omega;\br^3)$, it follows readily from the Lax-Milgram Theorem
 that the elliptic system (\ref{max-equation}) has a unique
(weak) solution in
\begin{equation}\label{V}
V^2_0(\Omega)=\big\{ u \in L^2(\Omega; \br^3): \ \nabla \times u \in L^2(\Omega;\br^3) \text{ and } n \times u =0 \text{ on }
\partial\Omega \big\}.
\end{equation}
Moreover, the solution satisfies the estimate
\begin{equation}\label{L-2-estimate}
\| u_\varep \|_{L^2(\Omega)} +\|\nabla\times u_\varep \|_{L^2(\Omega)}
\le C \left\{  \| F\|_{L^2(\Omega)} +\| G\|_{L^2(\Omega)} \right\},
\end{equation}
where $C$ depends only on $\mu$ and $\Omega$.
Suppose, in addition, that  the matrices $A(y)$ and $B(y)$
are periodic with respect to $\mathbb{Z}^3$:
\begin{equation}\label{periodicity}
A(y+z)=A(y) \quad \text{ and } \quad B(y+z)=B(y) \quad \text{ for any } y\in \br^3, z\in \mathbb{Z}^3.
\end{equation}
It follows from the theory of homogenization that
 $u_\varep \to u_0$ weakly in $V^2_0(\Omega)$ as $\varep\to 0$, and $u_0$ is the unique solution in $V^2_0(\Omega)$ of
 the homogenized system:
 \begin{equation}\label{max-equation-0}
\nabla \times \big(A_0 \nabla \times u_0\big) + B_0 u_0  = F +\nabla \times G  \quad  \text{ in } \Omega,\\
\end{equation}
where $A_0$ and $B_0$ are constant (effective) matrices given by
$$
A_0 =  \big( \mathcal{H}(A^{-1})\big)^{-1} \quad \text{ and } \quad B_0 =\mathcal{H} (B).
$$
We refer the reader to \cite[pp.81-91]{Bensoussan-1978} for the definition of $A_0$ and $B_0$
as well as the homogenization theory for (\ref{max-equation}).

In this paper we consider the boundary value problem for the elliptic system (\ref{max-equation}):
\begin{equation}\label{BVP}
\left\{
\aligned
\nabla \times \big(A(x/\varep) \nabla \times u_\varep\big) + B(x/\varep)u_\varep  & = F +\nabla \times G  &\quad&  \text{ in } \Omega,\\
n\times u_\varep & = f &\quad &\text{ on } \partial\Omega.
\endaligned
\right.
\end{equation}
We shall be interested in the estimates of $u_\varep$ and $\nabla \times u_\varep$, which are uniform in $\varep>0$,
in $L^p(\Omega)$ for $1<p\le \infty$, under the ellipticity and periodicity conditions on $A$ and $B$.

For $1<p<\infty$, let
\begin{equation}\label{V-p}
V^p(\Omega) =\big\{ u\in L^p(\Omega; \br^3): \, \nabla \times u\in L^p(\Omega;\br^3) \big\}.
\end{equation}
If $f\in L^p(\partial\Omega;\br^3)$ and $n\cdot f=0$ on $\partial\Omega$,
we will use Div$(f)$ to denote the surface
divergence of $f$ on $\partial\Omega$, defined by
\begin{equation}\label{surface}
<\text{Div} (f), \psi>_{W^{-1,p}(\partial\Omega)\times W^{1,p^\prime}(\partial\Omega)}
 =-\int_{\partial\Omega} <f, \nabla_{\tan} \psi>\, d\sigma,
\end{equation}
where $\psi\in C^\infty(\br^d)$ and $\nabla_{\tan} \psi =\nabla \psi -<\nabla\psi, n> n$  denotes the
tangential gradient of $\psi$ on $\partial\Omega$.
The following are the main results of the paper.

\begin{thm}\label{theorem-A}
Let $1<p<\infty$ and $\Omega$ be a bounded, simply connected, $C^{1, \alpha}$ domain in $\br^3$ with connected boundary.
Suppose that $A$ and $B$ satisfy conditions (\ref{ellipticity}) and (\ref{periodicity}) and
that $A, B$ are H\"older continuous.
Let $F, G\in L^p(\Omega; \br^3)$ and $f\in L^p(\partial\Omega;\br^3)$ with
$n\cdot f=0$ on $\partial\Omega$ and \text{\rm Div}$(f)\in W^{-\frac{1}{p},p}(\partial\Omega)$.
Then the boundary value problem (\ref{BVP}) has a unique solution in $V^p(\Omega)$.
Moreover, the solution $u_\varep$ satisfies
\begin{equation}\label{L-p-estimate}
\aligned
 \|u_\varep\|_{L^p(\Omega)}
& +\|\nabla\times u_\varep\|_{L^p(\Omega)}\\
&\le C_p\left\{ \| F\|_{L^p(\Omega)} +\|G\|_{L^p(\Omega)}
+ \| f\|_{L^p(\partial\Omega)} +\|\text{\rm Div}(f)\|_{W^{-\frac{1}{p}, p} (\partial\Omega)} \right\},
\endaligned
\end{equation}
where the constant $C_p$ is independent of $\varep>0$.
\end{thm}

\begin{thm}\label{theorem-B}
Let $\Omega$ be a bounded, simply connected, $C^{1, \alpha}$ domain in $\br^3$ with connected boundary.
Suppose that $A$ and $B$  satisfy conditions (\ref{ellipticity}) and (\ref{periodicity}) and
that $A$, $B$ are H\"older continuous. Also assume that $A$ is symmetric.
Let $F, G\in C^\gamma (\Omega;\br^3)$ and $f\in C^\gamma (\partial\Omega;\br^3)$ with $n\cdot f=0$ on $\partial\Omega$
and \text{\rm Div}$(f)\in C^\gamma (\partial\Omega)$ for some $\gamma>0$.
Let $u_\varep$ be the unique solution of (\ref{BVP}) in $V^2 (\Omega)$.
Then $u_\varep, \nabla \times u_\varep \in  L^\infty(\Omega;\br^3)$, and
\begin{equation}\label{L-infty-estimate}
\aligned
 \|u_\varep\|_{L^\infty(\Omega)}
& +\|\nabla\times u_\varep\|_{L^\infty(\Omega)}\\
&\le C_\gamma \left\{ \| F\|_{C^\gamma(\Omega)} +\|G\|_{C^\gamma(\Omega)}
+ \| f\|_{C^\gamma (\partial\Omega)} +\|\text{\rm Div}(f)\|_{C^\gamma(\partial\Omega)} \right\},
\endaligned
\end{equation}
where the constant $C_\gamma $ is independent of $\varep>0$.
\end{thm}

Besides the interest in their own rights,
uniform regularity estimates are an important tool in  the study of convergence problems for solutions $u_\varep$,
eigenfunctions, and eigenvalues in the theory of homogenization.
For  the elliptic systems
\begin{equation}\label{elliptic-system}
-\text{\rm div} \big( A(x/\varep)\nabla u_\varep \big) =F \quad \text{ in } \Omega,
\end{equation}
where $A(y) =\big( a_{ij}^{\alpha\beta}(y)\big)$ with $1\le i, j\le d$ and $1\le \alpha, \beta\le m$
is uniform elliptic, periodic, and H\"older continuous,
uniform $W^{1,p}$ estimates, H\"older estimates, and Lipschitz estimates
were established in \cite{AL-1987} \cite{AL-1991} for solutions in $C^{1,\alpha}$ domains
with the Dirichlet boundary condition.
Analogous results for solutions in $C^{1,\alpha}$ domains with the Neumann boundary conditions
were recently obtained in \cite{kls1}. We mention that for suitable solutions of
$\text{div} \big(A(x/\varep)\nabla u_\varep \big) =0$ in a Lipschitz domain $\Omega$,
under the additional symmetry condition
$a_{ij}^{\alpha\beta} (y)=a_{ji}^{\beta\alpha} (y)$,
the following
uniform $L^2$ Rellich estimates:
\begin{equation}\label{Rellich-estimate}
\left\|\frac{\partial u_\varep}{\partial \nu_\varep} \right\|_{L^2(\partial\Omega)}
\approx \| \nabla_{\tan} u_\varep\|_{L^2(\partial\Omega)}
\end{equation}
were proved in \cite{ks1} \cite{ks2}, where $\partial u_\varep/\partial\nu_\varep$ and $\nabla_{\tan} u_\varep$
denote the conormal derivative  and tangential gradient of $u_\varep$ on $\partial\Omega$, respectively.
The proof for the Lipschitz estimates
in \cite{kls1} relies on the $L^2$ Relllich estimates in \cite{ks2}. As a result, the Lipschitz
estimates in \cite{kls1} for solutions with the Neumann boundary conditions, which
are used in the proof of Theorem \ref{theorem-B}, were established under the
additional symmetry condition.

To prove Theorems \ref{theorem-A} and \ref{theorem-B}, our basic idea is to reduce the study of (\ref{max-equation})
to that of a scalar uniform elliptic equation of divergence form.
This uses the well-known fact that on a simply connected domain $\Omega$ in $\br^3$, $u\in L^2(\Omega;\br^3)$ and
$\nabla \times u=0$ in $\Omega$ imply that $u=\nabla P$ in $\Omega$ for some
scalar function $P\in H^1(\Omega)$. It also relies on the fact that on a bounded $C^1$ domain $\Omega$ with connected boundary,
$u\in L^p(\Omega;\br^3)$ and div$(u)=0$ in $\Omega$  imply that $u=\nabla\times v$ in $\Omega$ for some $v\in W^{1, p}(\Omega;\br^3)$.
The approach allows us to reduce the estimates (\ref{L-p-estimate}) and (\ref{L-infty-estimate})
to the $W^{1,p}$ and Lipschitz estimates for solutions of the scalar elliptic equation
\begin {equation}\label{elliptic-equation-plus}
-\text{\rm div} \big( A(x/\varep)( \nabla w_\varep +g) \big) =F \quad \text{ in } \Omega.
\end{equation}
We point out that both the Dirichlet condition and the Neumann condition for
the elliptic equation (\ref{elliptic-equation-plus}) are needed to handle the system (\ref{max-equation}).

The rest of the paper is organized as follows.
In Section 2 we collect some basic facts related to the divergence and curl operators,
which will be needed in Section 4.
In Section 3 we establish the $W^{1,p}$ and Lipschitz estimates for (\ref{elliptic-equation-plus}) in
a bounded $C^{1, \alpha}$ domain.
While the $W^{1, p}$ estimates for (\ref{elliptic-equation-plus}) follow readily from those for
(\ref{elliptic-system})  with $m=1$ in \cite{AL-1987, AL-1991}  and \cite{kls1},
the desired Lipschitz estimates require some additional argument, involving the Green and Neumann functions for (\ref{elliptic-system}).
The proof of Theorem \ref{theorem-A} is given in Section 4, and  the proof of Theorem \ref{theorem-B} in Section 5.
Finally, we point out that under the additional assumption that
$A$ is Lipschitz continuous, $B$ is a constant matrix, and $\Omega$ is
$C^{1,1}$, it follows from the estimate (\ref{L-p-estimate}) that
\begin{equation}\label{L-p-estimate-1}
\|u_\varep\|_{W^{1,p}(\Omega)}
\le
C_p\left\{ \| F\|_{L^p(\Omega)} +\|\text{\rm div} (F)\|_{L^p(\Omega)}
+\|G\|_{L^p(\Omega)}
+ \| f\|_{W^{1-\frac{1}{p}, p} (\partial\Omega)} \right\}
\end{equation}
for $1<p<\infty$ (see Remark \ref{remark-1.0}).

%%%%%%%%%%%%%%%%%%%%%%%%%%%%%%%%%%%%

\section{Some preliminaries}

The materials in this section are more or less known.

\begin{thm}\label{curl-theorem}
Let $\Omega$ be a bounded, simply-connected, Lipschitz domain in $\br^3$.
Suppose that $u\in L^p (\Omega;\br^3)$ for some $1<p<\infty$ and $\nabla\times u=0$ in $\Omega$.
Then $u=\nabla P$ in $\Omega$ for some $P\in W^{1,p}(\Omega)$.
\end{thm}

\begin{proof}
The case $p>2$ follows directly from the case $p=2$, which is well known.
The case $p<2$ may be proved in the same manner as in the case $p=2$
(see e.g. \cite[pp.31-32]{Girault-Raviart}).
\end{proof}

We will use $W^{t, p}(\partial\Omega)$ to denote the Sobolev-Besov space of order $t$ and exponent $p$ on $\partial\Omega$
for $-1<t<1$ and $1<p<\infty$. Note that the dual of $W^{t, p}(\partial\Omega)$ is given
by $W^{-t, q}(\partial\Omega)$, where $q=p^\prime=\frac{p}{p-1}$.

\begin{thm}\label{lemma-2.1}
Let $\Omega$ be a bounded $C^1$ domain in $\br^3$ with connected boundary.
Let $ g\in L^p(\Omega; \br^3)$ for some
$1<p<\infty$.
Suppose that $\text{\rm div} (g) =0$ in $\Omega$.
Then there exists $h\in W^{1, p}(\Omega; \br^3)$ such that
$\nabla \times h=g$ in $\Omega$. Moreover, $\text{\rm div} (h)=0$ in $\Omega$ and
$\| h\|_{W^{1,p} (\Omega)} \le C_p\, \| g\|_{L^p(\Omega)}$,
where $C_p$ depends only on $p$ and $\Omega$.
\end{thm}

\begin{proof}
The result is well known for smooth domains. The proof for the case of $C^1$ domains
is similar. We provide a proof, which follows the lines in \cite{Girault-Raviart} and \cite{Costabel-1990},
 for the sake of completeness.

We first note that if $u\in L^p(\br^3;\br^3)$ with supp$(u)\subset B=B(0,R)$ and $\text{div} (u)=0$ in $\br^3$,
then there exists $v\in W^{1, p}(B;\br^3)$ such that $\nabla \times v =u$ in $B$,
$\text{div} (v)=0$ in $B$, and $\| v\|_{W^{1,p}(B)} \le C_p \, \| u\|_{L^p(\br^3)}$.
To see this, we let $v=\nabla \times w$, where
$$
w(x)=\int_{\br^3} \Gamma (x-y) u(y)\, dy,
$$
and $\Gamma (x)=(4\pi |x|)^{-1}$ is the fundamental solution for $-\Delta$ in $\br^3$, with pole at the origin.
It follows from $\text{div} (u)=0$ in $\br^3$ that $\text{div} (w)=0$ in $\br^3$. Hence,
$$
\nabla \times v =\nabla \times (\nabla \times w)
=-\Delta w +\nabla (\text{div} (w) )=-\Delta w =u.
$$
Clearly, div$(v)=0$ in $\br^3$. Also, by the Calder\'on-Zygmund estimate and
fractional intergal estimate,
$$
\| v\|_{W^{1,p}(B)} \le C\, \|\nabla w\|_{W^{1,p}(B)} \le C_p \, \| u\|_{L^p(\br^3)},
$$
where $C_p$ may depend on $R$.

We now consider the case where $\Omega$ is a bounded $C^1$ domain with connected boundary.
Choose a ball $B=B(0, R)$ such that $\overline{\Omega} \subset B(0, R/4)$.
Since $\partial\Omega$ is connected, $B\setminus \overline{\Omega}$ is a bounded (connected) $C^1$ domain.
Also, $g\in L^p(\Omega;\br^3)$ and div$(g)=0$ in $\Omega$ imply that $n\cdot g\in W^{-\frac{1}{p}, p} (\partial\Omega)$ and
$$
<n\cdot g, 1>_{W^{-\frac{1}{p}, p}(\partial\Omega)\times W^{\frac{1}{p}, p^\prime}(\partial\Omega)}=0.
$$
 It follows from \cite{Fabes-1998} that there exists $f\in W^{1, p}(B\setminus \overline{\Omega})$ such that
$\Delta f =0 $ in $B\setminus \overline{\Omega}$,
$\frac{\partial f}{\partial n} = n\cdot g$ on $\partial\Omega$, and $\frac{\partial f}{\partial n}=0$ on $\partial B$.
Moreover,
 \begin{equation}\label{2.1.1}
 \|\nabla f\|_{L^p(B\setminus \overline{\Omega})} \le C_p \, \| n\cdot g\|_{W^{-\frac{1}{p}, p} (\partial\Omega)}
 \le C_p\, \| g\|_{L^p(\Omega)}.
 \end{equation}
Define
$$
\widetilde{g}=\left\{ \aligned
& g & \quad & \text{ in } \Omega,\\
& \nabla f & \quad & \text{ in } B \setminus \overline{\Omega},\\
& 0 & \quad & \text{ in } \br^3\setminus B.
\endaligned
\right.
$$
Note that $\widetilde{g}\in L^p(\br^3;\br^3)$ and for any $\psi\in C_0^\infty(\br^3)$,
$$
\aligned
\int_{\br^3} \widetilde{g}\cdot \nabla \psi \, dx
& =\int_{\Omega} g\cdot \nabla \psi \, dx + \int_{B\setminus \overline{\Omega}} \nabla f\cdot \nabla \psi\, dx\\
&=<n\cdot g, \psi>_{W^{-\frac{1}{p}, p} (\partial\Omega)\times W^{\frac{1}{p}, p^\prime} (\partial\Omega)}
+\int_{B\setminus \overline{\Omega}} \nabla f\cdot \nabla \psi\, dx\\
&=0.
\endaligned
$$
Thus, div$(\widetilde{g})=0$ in $\br^3$.
It follows from the first part of the proof that $\widetilde{g}=\nabla\times h$ in $B$ for some $h\in W^{1, p}(B;\br^3)$ with
div$(h)=0$ in $B$. Furthermore,
$$
\| h\|_{W^{1,p}(\Omega)}\le \| h\|_{W^{1, p}(B)}
 \le C_p \, \| \widetilde{g}\|_{L^p(B)} \le C_p \, \| g\|_{L^p(\Omega)},
$$
where we have used (\ref{2.1.1}) for the last inequality. This competes the proof.
\end{proof}

\begin{thm}\label{gradient-theorem}
Let $1<p<\infty$ and $\Omega$ be a bounded, simply-connected, $C^{1,1}$ domain in $\br^3$
with connected boundary.
Let $A=A(x)$ be a $3\times 3$ matrix in $\br^3$ satisfying the ellipticity condition (\ref{ellipticity-1}).
Also assume that $A$ is Lipschitz continuous.
Then, for any $u\in L^p(\Omega;\br^3)$ such that the right hand side of (\ref{gradient-estimate}) is finite,
\begin{equation}\label{gradient-estimate}
\|\nabla u\|_{L^p(\Omega)}
 \le C_p \bigg\{
\|\nabla \times u\|_{L^p(\Omega)}
+\| \text{\rm div} (A u)\|_{L^p(\Omega)}
+\|n\times u\|_{W^{1-\frac{1}{p}, p} (\partial\Omega)} \bigg\},
\end{equation}
where $C_p$ depends only on $p$, $\Omega$, and $A$.
\end{thm}

\begin{proof}
Let $u$ be a function in $L^p(\Omega;\br^3)$ such that the right hand side of (\ref{gradient-estimate}) is finite.
Let $g=\nabla \times u$ in $\Omega$.
Then $g\in L^p(\Omega;\br^3)$ and div$(g)=0$ in $\Omega$.
In view of Theorem \ref{lemma-2.1}, there exists $h\in W^{1,p}(\Omega;\br^3)$ such that
$g=\nabla \times h$ in $\Omega$, div$(h)=0$ in $\Omega$, and
\begin{equation}\label{1.0.1}
\| h\|_{W^{1,p}(\Omega)} \le C_p\, \| \nabla\times u\|_{L^p(\Omega)}.
\end{equation}
Let $w=u-h$ in $\Omega$.
Note that $w\in L^p(\Omega;\br^3)$ and $\nabla \times w=0$ in $\Omega$.
It then follows from Theorem \ref{curl-theorem} that there exists $P\in W^{1,p}(\Omega)$ such that
$w=\nabla P$ in $\Omega$.

We now observe that
\begin{equation}\label{1.0.2}
\text{\rm div} (A\nabla P)
=\text{\rm div}(Au) -\text{\rm div}(A h)\in L^p(\Omega)
\end{equation}
and
\begin{equation}\label{1.0.3}
n\times \nabla P
=n\times u-n\times h\in W^{1-\frac{1}{p}, p} (\partial\Omega),
\end{equation}
where we have used the fact that $\Omega$ is $C^{1,1}$ and
\begin{equation}\label{1.0.4}
\| n\times h\|_{W^{1-\frac{1}{p}, p} (\partial\Omega)}
\le C \, \| h\|_{W^{1-\frac{1}{p}, p} (\partial\Omega)} \le C\, \| h\|_{W^{1, p}(\Omega)}
\le C\, \| \nabla \times u\|_{L^p(\Omega)}.
\end{equation}
Finally, we note that if $\int_{\partial\Omega} P=0$,
$$
\|P\|_{W^{2-\frac{1}{p}, p} (\partial\Omega)}
\le C\, \|\nabla_{\tan} P\|_{W^{1-\frac{1}{p}, p} (\partial\Omega)}
\le C\, \| n \times \nabla P\|_{W^{1-\frac{1}{p}, p} (\partial\Omega)}.
$$
It follows from the $W^{2,p}$ estimates for elliptic equations in $C^{1,1}$ domains (see e.g. \cite{Gilbarg-Trudinger})
that
\begin{equation}\label{1.0.5}
\aligned
\| u\|_{W^{1, p}(\Omega)}
&\le \| h\|_{W^{1, p}(\Omega)} +\| \nabla P\|_{W^{1,p}(\Omega)} \\
& \le C \left\{ \|\nabla \times u\|_{L^p(\Omega)}
+\|\text{\rm div} (A u)\|_{L^p(\Omega)}
+\| n\times u\|_{W^{1-\frac{1}{p}, p} (\partial\Omega)} \right\}.
\endaligned
\end{equation}
This completes the proof.
\end{proof}

\begin{remark}\label{remark-1.0}
{\rm
Assume that $\Omega$ is a bounded, simply-connected, $C^{1,1}$ domain in
$\br^3$ with connected boundary and that $B$ is a positive-definite constant matrix.
Let $u_\varep$ be a solution of (\ref{BVP}).
It follows from (\ref{gradient-estimate}) that
$$
\aligned
\|\nabla u_\varep\|_{W^{1, p}(\Omega)}
& \le C_p \left\{ \|\nabla \times u_\varep\|_{L^p(\Omega)}
+\|\text{\rm div} (Bu_\varep)\|_{L^p(\Omega)}
+\| n\times u_\varep\|_{W^{1-\frac{1}{p}, p} (\partial\Omega)} \right\}\\
&=
C_p \left\{ \|\nabla \times u_\varep\|_{L^p(\Omega)}
+\|\text{\rm div} (F)\|_{L^p(\Omega)}
+\| n\times u_\varep\|_{W^{1-\frac{1}{p}, p} (\partial\Omega)} \right\}.
\endaligned
$$
This, together with (\ref{L-p-estimate}), gives (\ref{L-p-estimate-1}).
}
\end{remark}

%%%%%%%%%%%%%%%%%%%%%%%%%%%%%%%%%%%%%%%

%%%%%%%%%%%%%%%%

\section{Uniform estimates for scalar elliptic equations with periodic coefficients}

In this section we establish the $W^{1,p} $ and Lipschitz estimates for solutions of
the elliptic equation (\ref{elliptic-equation-plus}). These estimates will be used in the proof of Theorems \ref{theorem-A} and
 \ref{theorem-B}.

Let $A=A(y)=(a_{ij} (y))$ be a $d\times d$ matrix in $\br^d$, $d\ge 2$.
We say $A\in \Lambda(\mu, \lambda, \tau)$  for some $\mu>0$, $\tau\in (0,1]$, and $\lambda\ge 0$, if
$A$ satisfies the ellipticity condition,
\begin{equation}\label{ellipticity-1}
\mu |\xi|^2 \le a_{ij} (y) \xi_i \xi_j \le \frac{1}{\mu} |\xi|^2 \quad \text{ for any } y\in \br^d \text{ and } \xi \in \br^d,
\end{equation}
the periodicity condition,
\begin{equation}\label{periodicity-1}
A(y+z) =A(y) \quad \text{ for any } y\in\br^d \text{ and } z\in \mathbb{Z}^d,
\end{equation}
and the smoothness condition,
\begin{equation}\label{smoothness}
|A(x)-A(y)|\le \lambda |x-y|^\tau \text{ for any } x,y \in \br^d.
\end{equation}

We start out with the $W^{1,p}$ estimate for solutions of the Dirichlet problem.

\begin{thm}\label{theorem-2.0}
Let  $1<p<\infty$ and $\Omega$ be a bounded $C^{1,\alpha}$ domain in $\br^d$, $d\ge 2$.
Suppose that $A\in \Lambda (\mu, \lambda, \tau)$. Let $w_\varep \in W^{1, p}(\Omega)$ be the solution
of the Dirichlet problem:
\begin{equation}\label{Dirichlet-problem}
\left\{\aligned
\text{\rm div} \left\{ A(x/\varep) (\nabla w_\varep +g)\right\} & =\text{\rm div} (F) & \quad & \text{ in } \Omega,\\
w_\varep & =f& \quad & \text{ on } \partial\Omega,
\endaligned
\right.
\end{equation}
where $g\in L^p(\Omega;\br^d)$, $F\in L^p(\Omega;\br^d)$,  and $f\in W^{1-\frac{1}{p}, p}(\partial\Omega)$. Then,
\begin{equation}\label{estimate-2.0}
\| w_\varep\|_{W^{1, p}(\Omega)}
\le C_p \left\{ \| g\|_{L^p(\Omega)} + \| F\|_{L^p(\Omega)} +\| f\|_{W^{1-\frac{1}{p}, p} (\partial\Omega)}\right\},
\end{equation}
where $C_p$ depends only on $p$, $\mu$, $\lambda$, $\tau$, and $\Omega$.
\end{thm}

\begin{proof} Rewrite the elliptic equation in (\ref{Dirichlet-problem}) as
\begin{equation}\label{2.0.1}
\text{div} \left\{ A(x/\varep)\nabla w_\varep\right\} =\text{div} \big( F-A(x/\varep) g\big).
\end{equation}
The estimate (\ref{estimate-2.0}) is a simple consequence of \cite[Theorem C]{AL-1991}.
\end{proof}

The next theorem establishes the Lipschitz estimate for solutions of the
Dirichlet problem.

\begin{thm}\label{theorem-2.2}
Suppose that $A$ and $\Omega$ satisfy the same assumptions as in Theorem \ref{theorem-2.0}.
Let $g\in C^\gamma (\Omega; \br^d)$, $F\in C^\gamma (\Omega; \br^d)$, and $f\in C^{1,\gamma}(\partial\Omega)$ for some $0<\gamma<\alpha$.
Let $w_\varep\in H^1(\Omega)$ be the solution of the Dirichlet problem (\ref{Dirichlet-problem}).
Then $\nabla w_\varep \in L^\infty(\Omega)$ and
\begin{equation}\label{estimate-2.2}
\|\nabla w_\varep\|_{L^\infty(\Omega)}
\le C_\gamma\left\{ \| g\|_{C^\gamma (\Omega)} +\| F\|_{C^\gamma (\Omega)}
+\| f \|_{C^{1, \gamma}(\partial\Omega)} \right\},
\end{equation}
where $C_\gamma$ depends only on $\gamma$, $\mu$, $\lambda$, $\tau$, and $\Omega$.
\end{thm}

\begin{proof}
We begin by choosing $h \in C^{1,\gamma}(\overline{\Omega})$ so that
$h=f$ on $\partial\Omega$ and $\| h \|_{C^{1,\gamma}(\Omega)} \le C\, \| f\|_{C^{1, \gamma}(\partial\Omega)}$.
By considering $w_\varep-h $, we may assume that $f=0$.
Next, in view of (\ref{2.0.1}), we may write
\begin{equation}\label{2.2.1}
\aligned
w_\varep (x)  &=-\int_\Omega \frac{\partial}{\partial y_i} \big\{ G_\varep (x, y)\big\} a_{ij}(y/\varep) g_j (y)\, dy
+\int_\Omega \frac{\partial}{\partial y_i} \big\{ G_\varep (x, y)\big\} F_i (y)\, dy\\
&= w_\varep^{(1)} (x) +w_\varep^{(2)} (x) ,
\endaligned
\end{equation}
where $F=(F_1, \dots, F_d)$ and $G_\varep (x,y)$ denotes the Green function for the operator $-\text{div} (A(x/\varep)\nabla)$
in $\Omega$, with pole at $y$. It follows from \cite{AL-1987} that for any $x,y\in \Omega$,
\begin{equation}\label{Green-function-estimate}
\aligned
|G_\varep(x,y)| & \le C |x-y|^{2-d},\\
|\nabla_x G_\varep (x,y)| +|\nabla_y G_\varep (x,y)| &\le C |x-y|^{1-d},\\
|\nabla_x\nabla_y G_\varep(x,y)| &\le C |x-y|^{-d},
\endaligned
\end{equation}
where $C$ depends only on $\mu$, $\lambda$, $\tau$, and $\Omega$.
We note that if $d=2$, the first inequality in (\ref{Green-function-estimate})
should be replaced by $|G_\varep (x,y)|\le C (1+\log |x-y|)$.
Using (\ref{Green-function-estimate}), we see that for any $x\in \Omega$,
\begin{equation}\label{2.2.3}
\aligned
|\nabla w_\varep^{(2)} (x)|
& =\left|\int_\Omega \frac{\partial}{\partial y_i} \big\{ \nabla_x G_\varep (x,y) \big\} \big\{ F_i (y) -F_i(x) \big\} \, dy\right|\\
&\le C\, \| F\|_{C^\gamma(\Omega)}
\int_\Omega \frac{dy}{|x-y|^{d-\gamma}}\\
&\le C \, \| F\|_{C^\gamma(\Omega)}.
\endaligned
\end{equation}

Finally, to estimate $\nabla w_\varep^{(1)}$, we let $\Phi_\varep (x)$ be the Dirichlet corrector for the
operator $-\text{div}(A(x/\varep)\nabla)$ in $\Omega$; i.e., $\Phi_\varep
=(\Phi_{\varep, 1} (x), \dots, \Phi_{\varep, d} (x) )$ is  the function in  $ H^1(\Omega;\br^d)$
satisfying
\begin{equation}\label{Dirichlet-corrector}
\left\{
\aligned
\text{div} \big(A(x/\varep)\nabla \Phi_{\varep, k}  \big) & =0& \quad & \text{ in } \Omega,\\
\Phi_{\varep, k} & =x_k &  \quad & \text{ on } \partial\Omega.
\endaligned
\right.
\end{equation}
Since $\Phi_{\varep, k} -x_k=0$ on $\partial\Omega$ and
$$
-\text{div} \big\{ A(x/\varep) \nabla \big(\Phi_{\varep, k} -x_k\big)\big\}
=\frac{\partial}{\partial x_i} \big\{  a_{ik} (x/\varep)\big\}\quad \text{ in } \Omega,
$$
we see that
\begin{equation}\label{2.2.5}
\Phi_{\varep, k} (x) -x_k =-\int_\Omega \frac{\partial}{\partial y_i}
\big\{ G_\varep (x,y) \big\} a_{ik} (y/\varep)\, dy
\end{equation}
for $1\le k\le d$.
It follows that
$$
\aligned
|\nabla w_\varep^{(1)} (x)|
& =\left| \int_\Omega \frac{\partial}{\partial y_i} \big\{ \nabla_x G_\varep(x,y) \big\} \big\{ a_{ij} (y/\varep) g_j(y) -a_{ij}(x/\varep) g_j (x) \big\}\, dy \right|\\
&\le  \int_\Omega |\nabla_x\nabla_y G_\varep(x,y)|\, |A(y/\varep)|\, |g(y)-g(x)|\, dy\\
& \qquad +\left| g_j (x) \int_\Omega \frac{\partial}{\partial y_i} \big\{ \nabla_x G_\varep (x,y)\big\} \big\{ a_{ij} (y/\varep)
-a_{ij} (x/\varep) \big\} \, dy \right|\\
&\le C\, \| g\|_{C^\gamma (\Omega)}
+|g_j (x)|\, |\nabla \left\{ \Phi_{\varep, j} (x) -x_j\right\} |,
\endaligned
$$
where we have used (\ref{2.2.5}) and
the estimate $|\nabla_x\nabla_y G_\varep(x,y)|\le C |x-y|^{-d}$
for the last inequality.
This, together with the Lipschitz estimate $\|\nabla \Phi_\varep\|_{L^\infty(\Omega)} \le C$,
established in \cite{AL-1987}, yields
$
\|\nabla w_\varep^{(1)}\|_{L^\infty(\Omega)}
\le C\, \| g\|_{C^\gamma(\Omega)}
$.
The proof is complete.
\end{proof}

We now turn to the $W^{1,p}$ estimate for solutions of the Neumann problem.

\begin{thm}\label{theorem-2.3}
Let $1<p<\infty$ and $\Omega$ be a bounded $C^{1, \alpha}$ domain in $\br^d$, $d\ge 2$.
Suppose that $A\in \Lambda (\mu, \lambda, \tau)$.
Let $w_\varep \in W^{1, p}(\Omega)$ be a solution of the Neumann problem:
\begin{equation}\label{Neumann-problem}
\left\{\aligned
\text{\rm div} \left\{ A(x/\varep) (\nabla w_\varep +g)\right\} & =0 & \quad & \text{ in } \Omega,\\
n\cdot A(x/\varep) (\nabla w_\varep +g) & =f & \quad & \text{ on } \partial\Omega,
\endaligned
\right.
\end{equation}
where $g \in L^p(\Omega;\br^d)$, $f\in W^{-\frac{1}{p}, p}(\partial\Omega)$ and $<f, 1>=0$.
 Then
\begin{equation}\label{estimate-2.3}
\| \nabla w_\varep \|_{L^p(\Omega)} \le C_p \left\{ \| g\|_{L^p(\Omega)} +\| f\|_{W^{-\frac{1}{p}, p} (\partial\Omega)} \right\},
\end{equation}
where $C_p$ depends only on $p$, $\mu$, $\lambda$, $\tau$, and $\Omega$.
\end{thm}

\begin{proof}
This is a direct consequence of Theorem 1.1 in  \cite{kls1}.
\end{proof}

The next theorem gives the Lipschitz estimate for solutions of the Neumann problem
(\ref{Neumann-problem}).
Note that in addition to the ellipticity and periodicity conditions,
we also assume that $A^*=A$; i.e., $a_{ij}(y)=a_{ji}(y)$.

\begin{thm}\label{theorem-2.4}
Let $\Omega$ be a bounded $C^{1, \alpha}$ domain in $\br^d$, $d\ge 2$.
Suppose that $A\in \Lambda (\mu, \lambda, \tau)$ and $A^*=A$.
Let $g\in C^{\gamma} (\Omega;\br^d)$, and $f\in C^\gamma(\partial\Omega)$ with mean value zero, for some $0<\gamma<\alpha$.
Let $w_\varep \in H^1(\Omega)$ be a solution of the Neumann problem (\ref{Neumann-problem}).
Then $\nabla u_\varep \in L^\infty(\Omega)$, and
\begin{equation}\label{estimate-2.4}
\| \nabla u_\varep \|_{L^\infty(\Omega)}
\le C_\gamma \left\{  \| g\|_{C^\gamma (\Omega)} +\| f\|_{C^\gamma (\partial\Omega)}\right\},
\end{equation}
where $C_\gamma $ depends only on $\gamma$, $\mu$, $\lambda$, $\tau$, and $\Omega$.
\end{thm}

\begin{proof}
Let $v_\varep \in H^1(\Omega)$ be a solution of the Neumann problem:
$\text{div}\big(A(x/\varep)\nabla v_\varep\big)=0$ in $\Omega$ and $n\cdot A(x/\varep)\nabla v_\varep =f$ on $\partial\Omega$.
It follows from \cite[Theorem 1.2]{kls1} that
$$
\|\nabla v_\varep \|_{L^\infty(\Omega)} \le C\, \| f\|_{C^\gamma(\partial\Omega)},
$$
where $C$ depends only on $\gamma$, $\mu$, $\lambda$, $\tau$, and $\Omega$.
Thus, by considering $w_\varep-v_\varep$, we may assume that $f=0$.

Let $g=(g_1, \dots, g_d)\in C^\gamma(\Omega;\br^d)$ and $w_\varep$ be a solution of (\ref{Neumann-problem}) with $f=0$. Then
$$
w_\varep (x)= -\int_\Omega \frac{\partial}{\partial y_i} \big\{ N_\varep(x,y) \big\} a_{ij}(y/\varep) g_j (y)\, dy
 + E
$$
for some constant $E$,
where $N_\varep(x,y)$ denotes the Neumann function for the elliptic operator $-\text{div}(A(x/\varep)\nabla)$
in $\Omega$, with pole at $y$. Under the assumption that $A\in \Lambda (\mu, \lambda,\tau)$ and
$A^*=A$, it was proved in \cite{kls1} that for $d\ge 3$,
\begin{equation}\label{Neumann-function-estimate}
\aligned
|N_\varep(x,y) & \le C |x-y|^{2-d},\\
|\nabla_x N_\varep (x,y)| +|\nabla_y N_\varep (x,y)| &\le C |x-y|^{1-d},\\
|\nabla_x\nabla_y N_\varep(x,y)| &\le C |x-y|^{-d},
\endaligned
\end{equation}
where $C$ depends only on $\mu$, $\lambda$, $\tau$, and $\Omega$.
If $d=2$, one obtains $|N_\varep (x,y)|\le C_\eta |x-y|^{-\eta}$,
$|\nabla_x N_\varep(x,y)| +|\nabla_y N_\varep (x,y)|\le C_\eta |x-y|^{-1-\eta}$,
and $|\nabla_x\nabla_y N_\varep (x,y)|\le C_\eta |x-y|^{-2-\eta}$
for any $\eta>0$ (this is not sharp, but enough for the proof of this theorem).
It follows that for any $x\in \Omega$,
\begin{equation}\label{2.4.1}
\aligned
\nabla w_\varep (x)
&=-\int_\Omega \frac{\partial}{\partial y_i} \big\{ \nabla_x N_\varep (x,y) \big\} \big[ a_{ij}(y/\varep) g_j (y)
-a_{ij}(x/\varep) g_j(x) \big]\, dy\\
& \qquad
-a_{ij}(x/\varep) g_j(x) \int_{\partial\Omega} n_i (y) \nabla_x N_\varep (x,y)\, d\sigma (y)\\
&= -\int_\Omega \frac{\partial}{\partial y_i} \big\{ \nabla_x N_\varep (x,y) \big\} \big[ g_j (y)
- g_j(x) \big] a_{ij} (y/\varep)\, dy\\
&\qquad -g_j(x) \int_\Omega \frac{\partial}{\partial y_i} \big\{ \nabla_x N_\varep (x,y) \big\} \big[
a_{ij} (y/\varep) -a_{ij} (x/\varep) \big]\, dy\\
&\qquad
-a_{ij}(x/\varep) g_j (x) \int_{\partial\Omega} n_i (y) \nabla_x N_\varep (x,y)\, d\sigma (y).
\endaligned
\end{equation}
Note that if $g_j(x)=-\delta_{jk}$, then $w_\varep (x) =x_k$ is a solution of (\ref{Neumann-problem}) with $f=0$.
In view of (\ref{2.4.1}), this implies that
\begin{equation}\label{2.4.3}
\aligned
\nabla (x_k)
&=\delta_{jk} \int_\Omega \frac{\partial}{\partial y_i} \big\{ \nabla_x N_\varep (x,y)\big\}
\big[ a_{ij}(y/\varep) -a_{ij} (x/\varep) \big]\, dy\\
&\qquad +a_{ij}(x/\varep) \delta_{jk} \int_{\partial\Omega}
n_i(y) \nabla_x N_\varep (x,y)\, d\sigma (y).
\endaligned
\end{equation}
By combining (\ref{2.4.1}) and (\ref{2.4.3}) we obtain
$$
\nabla w_\varep (x) + g_j (x)\nabla (x_j)
=-\int_\Omega \frac{\partial}{\partial y_i} \big\{ \nabla_x N_\varep (x,y) \big\} \big[ g_j (y)
- g_j(x) \big] a_{ij} (y/\varep)\, dy.
$$
As a result, for any $x\in \Omega$,
$$
\aligned
|\nabla w_\varep (x)|
&\le C \|g\|_{L^\infty(\Omega)}
+C\, \| g\|_{C^\gamma (\Omega)}  \int_\Omega \frac{dy}{|x-y|^{d-\gamma}}\\
&\le C \| g\|_{C^\gamma(\Omega)},
\endaligned
$$
where we have used the estimate $|\nabla_x\nabla_y N_\varep (x,y)|\le C |x-y|^{-d}$
(the case $d=2$ may be handled in a similar manner).
This finishes the proof.
\end{proof}

%%%%%%%%%%%%%%%%%%%%%%%%%%%%%%%%%%

%%%%%%%%%%%%%%%%%%%%%%%%%%%%%%%%%%%

\section{$L^p$ estimates}

The goal of this section is to prove Theorem \ref{theorem-A}.
Throughout this section we will assume that $\Omega$ is a bounded, simply connected, $C^{1, \alpha}$ domain in $\br^3$ with
connected boundary, and that $A, B\in \Lambda (\mu, \lambda, \tau)$.

\begin{lemma}\label{lemma-3.1}
Let $2\le q<3$ and $2\le p\le p_0$, where $\frac{1}{p_0}=\frac{1}{q}-\frac{1}{3}$.
Given $F\in L^q(\Omega;\br^3)$ and $G\in L^p(\Omega;\br^3)$, let $u_\varep$ be the unique solution in $V_0^2(\Omega)$
of (\ref{BVP}) with $f=0$.
Suppose that $u_\varep \in L^q(\Omega;\br^3)$. Then $\nabla \times u_\varep \in L^p(\Omega;\br^3)$, and
\begin{equation}\label{estimate-3.1}
\| \nabla\times  u_\varep \|_{L^p(\Omega)}
\le C \left\{ \| F\|_{L^q(\Omega)} +\| G\|_{L^p(\Omega)} +\| u_\varep\|_{L^q(\Omega)} \right\},
\end{equation}
where $C$ depends only on $q$, $\mu$, $\lambda$, $\tau$, and $\Omega$.
\end{lemma}

\begin{proof}
It follows from  the elliptic system in (\ref{max-equation}) that
$$
\text{div} \big( B(x/\varep) u_\varep - F) =0 \quad \text{ in } \Omega.
$$
Since $B(x/\varep) u_\varep-F \in L^q(\Omega;\br^3)$,
by Theorem \ref{lemma-2.1}, there exists $h_\varep \in W^{1, q}(\Omega;\br^3)$ such that
\begin{equation}\label{3.1.1}
\nabla \times h_\varep =B(x/\varep)u_\varep -F \quad \text{ in } \Omega,
\end{equation}
and
\begin{equation}\label{3.1.3}
\| h_\varep \|_{W^{1,q}(\Omega)} \le C_q \, \| B(x/\varep) u_\varep -F\|_{L^q(\Omega)}.
\end{equation}
Thus,
$$
\nabla \times \big\{ A(x/\varep) \nabla \times u_\varep +h_\varep-G  \big\} =0 \quad \text{ in } \Omega.
$$
Since $\Omega$ is simply connected, there exists  $P_\varep \in W^{1,2}(\Omega)$ such that
\begin{equation}\label{3.1.5}
A(x/\varep)\nabla\times u_\varep +h_\varep -G =\nabla P_\varep \quad \text{ in } \Omega.
\end{equation}
It follows that
\begin{equation}\label{3.1.6}
A^{-1}(x/\varep) \nabla P_\varep =\nabla \times u_\varep +A^{-1}(x/\varep) \big( h_\varep
 -  G\big) \quad \text{ in } \Omega.
\end{equation}
Thus, $P_\varep\in W^{1,2}(\Omega)$ is the solution of
\begin{equation}\label{3.1.7}
\left\{\aligned
\text{\rm div} \big\{  A^{-1}(x/\varep)\big( \nabla P_\varep -h_\varep +G\big)\big\}
 & =0 &  \quad & \text{ in } \Omega,\\
 n\cdot  A^{-1}(x/\varep)\big( \nabla P_\varep -h_\varep +G\big)
 & = 0 & \quad & \text{ on } \partial\Omega,
 \endaligned
 \right.
 \end{equation}
 where we have used the fact that
 $$
 n\cdot (\nabla \times u_\varep) =-\text{Div} (n\times u_\varep) =0
 \quad \text{ on } \partial\Omega.
 $$
   In view of Theorem \ref{theorem-2.3}, we obtain
 $$
 \|\nabla P_\varep\|_{L^p(\Omega)}
 \le C\left\{ \| h_\varep\|_{L^p(\Omega)} +\| G\|_{L^p(\Omega)} \right\},
 $$
 where $C$ depends only on $p$, $\mu$, $\lambda$, $\tau$, and $\Omega$. This, together with
 (\ref{3.1.6}) and  the estimate (\ref{3.1.3}), gives
 \begin{equation}\label{3.1.9}
 \aligned
 \|\nabla \times u_\varep\|_{L^p(\Omega)}
 &\le C\left\{ \| h_\varep\|_{L^p(\Omega)} +\| G\|_{L^p(\Omega)} \right\}\\
  &  \le C \left\{ \| h_\varep\|_{W^{1, q}(\Omega)} +\| G\|_{L^p(\Omega)}\right\} \\
 &\le C \left\{ \| F\|_{L^q(\Omega)} +\| G\|_{L^p(\Omega)} +\| u_\varep\|_{L^q(\Omega)} \right\},
 \endaligned
 \end{equation}
 where we also used the Sobolev imbedding for the second inequality.
\end{proof}

\begin{remark}\label{remark-3.1}
{\rm
Let $F\in L^q(\Omega;\br^3)$, $G\in C^\gamma(\Omega;\br^3)$, and $f\in C^\gamma(\partial\Omega;\br^3)$
with $n\cdot f=0$ on $\partial\Omega$ and Div$(f)\in C^\gamma(\partial\Omega)$,
 where $3<q<\infty$ and $\gamma=1-\frac{3}{q}<\alpha$.
Let $u_\varep$ be the solution in $V^2(\Omega)$ of (\ref{BVP}).
Suppose that $A$ is symmetric.
Then
\begin{equation}\label{remark-3.1.1}
\|\nabla \times u_\varep\|_{L^\infty(\Omega)}
\le C \left\{ \| F\|_{L^q(\Omega)} +\| G\|_{C^\gamma(\Omega)} +\|\text{\rm Div} (f)\|_{C^\gamma(\partial\Omega)}
+\| u_\varep\|_{L^q(\Omega)} \right\},
\end{equation}
where $C$ depends only $q$, $\mu$, $\lambda$, $\tau$, and $\Omega$.
To see this, we let $h_\varep\in W^{1,q}(\Omega;\br^3)$ and $P_\varep\in W^{1,2}(\Omega)$
be the same functions as in the proof of Lemma \ref{lemma-3.1}.
It follows from  (\ref{3.1.5}), (\ref{3.1.7}) and Theorem \ref{theorem-2.4} that
$$
\aligned
\|\nabla \times u_\varep\|_{L^\infty(\Omega)} &  \le C \left\{ \|\nabla P_\varep\|_{L^\infty(\Omega)}
+\| h_\varep\|_{L^\infty(\Omega)} +\| G\|_{L^\infty(\Omega)} \right\}\\
&\le C \left\{ \| h_\varep\|_{C^\gamma(\Omega)} + \| G\|_{C^\gamma(\Omega)} +\|\text{\rm Div} (f)\|_{C^\gamma(\partial\Omega)} \right\} \\
&\le C \left\{ \| h_\varep\|_{W^{1,q}(\Omega)} +\| G\|_{C^\gamma (\Omega)}+\|\text{\rm Div} (f)\|_{C^\gamma(\partial\Omega)} \right\}\\
&\le C\left\{ \| F\|_{L^q(\Omega)} +\| G\|_{C^\gamma(\Omega)}
+\|\text{\rm Div} (f)\|_{C^\gamma(\partial\Omega)} +\| u_\varep\|_{L^q(\Omega)} \right\},
\endaligned
$$
where we have used the Sobolev imbedding for the third inequality and (\ref{3.1.3}) for the last.
}
\end{remark}

Next we reverse the roles of $u_\varep$ and $\nabla\times u_\varep$ in the estimate (\ref{estimate-3.1}).

\begin{lemma}\label{lemma-3.2}
Let $2\le q<3$ and $2\le p\le p_0$, where $\frac{1}{p_0}=\frac{1}{q}-\frac13$.
Given $F\in L^p(\Omega;\br^3)$ and $G\in L^q(\Omega;\br^3)$, let $u_\varep$ be the unique solution in $V_0^2(\Omega)$
of (\ref{BVP}) with $f=0$. Suppose that $\nabla \times u_\varep \in L^q(\Omega;\br^3)$.
Then $u_\varep \in L^p(\Omega;\br^3)$ and
\begin{equation}\label{estimate-3.2}
\| u_\varep \|_{L^p(\Omega)}
\le C\left\{ \| F\|_{L^p(\Omega)} +\| G\|_{L^q(\Omega)} +\| \nabla \times u_\varep\|_{L^q(\Omega)} \right\},
\end{equation}
where $C$ depends only on  $q$, $\mu$, $\lambda$, $\tau$, and  $\Omega$.
\end{lemma}

\begin{proof}
Let
\begin{equation}\label{3.2.1}
v_\varep =A(x/\varep) \nabla \times u_\varep -G \quad \text{ in } \Omega.
\end{equation}
Then $\nabla \times u_\varep =A^{-1}(x/\varep) \big( v_\varep +G\big)$ in $\Omega$.
It follows that
$$
\text{\rm div} \big( A^{-1}(x/\varep) (v_\varep +G)\big) =0 \quad \text{ in } \Omega.
$$
By Theorem \ref{lemma-2.1} there exists $h_\varep \in W^{1, q}(\Omega;\br^3)$ such that
\begin{equation}\label{3.2.3}
A^{-1}(x/\varep) ( v_\varep +G) =\nabla \times h_\varep \quad \text{ in } \Omega
\end{equation}
and
\begin{equation}\label{3.2.5}
\| h_\varep\|_{W^{1, q}(\Omega)}
\le C \left\{ \| v_\varep\|_{L^q(\Omega)} +\| G\|_{L^q(\Omega)} \right\}
\le C \left\{ \| \nabla \times u_\varep \|_{L^q(\Omega)} +\| G\|_{L^q(\Omega)} \right\}.
\end{equation}
Note that by the elliptic system  (\ref{max-equation}), $\nabla \times v_\varep =-B(x/\varep) u_\varep + F \text{ in } \Omega$. Thus,
\begin{equation}\label{3.2.7}
u_\varep =B^{-1}(x/\varep) F -B^{-1}(x/\varep) \nabla \times v_\varep \quad \text{ in } \Omega.
\end{equation}
Since $\Omega$ is simply connected and $\nabla \times h_\varep =\nabla \times u_\varep$ in $\Omega$,
there exists $Q_\varep\in W^{1,2}(\Omega)$ such that
$ \nabla Q_\varep =u_\varep -h_\varep $ in $\Omega$. Thus,
$$
B(x/\varep) \big( \nabla Q_\varep +h_\varep\big) = B(x/\varep) u_\varep
=F- \nabla\times  v_\varep  \quad \text{ in } \Omega.
$$
It follows that
\begin{equation}\label{3.2.9}
\text{div} \big\{  B(x/\varep) \big( \nabla Q_\varep +h_\varep\big)
\big\} =\text{div} \big( F\big) \quad \text{ in } \Omega.
\end{equation}
In view of Theorem \ref{theorem-2.0} we obtain
\begin{equation}\label{3.2.10}
\| \nabla Q_\varep\|_{L^p(\Omega)}
 \le C \left\{ \| h_\varep\|_{L^p(\Omega)} +\| F\|_{L^p(\Omega)} + \| Q_\varep \|_{W^{1-\frac{1}{p}, p} (\partial\Omega)} \right\},
\end{equation}
where $C$ depends only on $p$, $\mu$, $\lambda$, $\tau$, and $\Omega$.

Finally, we note that
\begin{equation}\label{3.2.11}
n\times \nabla Q_\varep =n\times u_\varep -n\times h_\varep =-n\times h_\varep \quad \text{ on } \partial\Omega.
\end{equation}
By subtracting a constant we may assume that $\int_{\partial\Omega} Q_\varep\, d\sigma =0$.
Since $|\nabla_{\tan} Q_\varep|=  | n\times \nabla Q_\varep|$ on $\partial\Omega$, we see that
$$
\aligned
\|Q_\varep\|_{W^{1-\frac{1}{p}, p}(\partial\Omega)}
&\le C \, \| n\times \nabla Q_\varep\|_{L^{q_1}(\partial\Omega)}
= C\, \| h_\varep\|_{L^{q_1}(\partial\Omega)}\\
& \le C\, \| h_\varep\|_{W^{1-\frac{1}{q}, q}(\partial \Omega)}
\le C \, \| h_\varep\|_{W^{1,q}(\Omega)}\\
& \le C \left\{ \|\nabla \times u_\varep\|_{L^q(\Omega)} +\| G\|_{L^q(\Omega)} \right\},
\endaligned
$$
where $q_1=2p/3$ and we have used the Sobolev imbedding on $\partial\Omega$ as well as
the estimate (\ref{3.2.5}). This, together with (\ref{3.2.10}) and (\ref{3.2.5}), gives
$$
\aligned
\| u_\varep\|_{L^p(\Omega)}
&\le \| \nabla Q_\varep\|_{L^p(\Omega)} +\| h_\varep\|_{L^p(\Omega)}\\
&\le C \left\{ \| F\|_{L^p(\Omega)} + \| G\|_{L^q(\Omega)} +\|\nabla \times u_\varep \|_{L^q(\Omega)} \right\},
\endaligned
$$
and completes the proof.
\end{proof}

\begin{remark}\label{remark-3.2}
{\rm
Let $F\in C^\gamma(\Omega;\br^3)$, $G\in L^q(\Omega;\br^3)$, and $f\in C^\gamma(\partial\Omega;\br^3)$ with $n\cdot f=0$
on $\partial\Omega$,
where $3<q<\infty$ and $\gamma=1-\frac{3}{q}<\alpha$.
Let $u_\varep$ be the solution in $V^2(\Omega)$ of (\ref{BVP}). Then
\begin{equation}\label{remark-3.2.1}
\|u_\varep\|_{L^\infty(\Omega)}
\le C\left\{ \| F\|_{C^\gamma(\Omega)} +\| G\|_{L^q(\Omega)} +\| f\|_{C^\gamma(\partial\Omega)}
+\| \nabla \times u_\varep\|_{L^q(\Omega)} \right\},
\end{equation}
where $C$ depends only on $q$, $\mu$, $\lambda$, $\tau$, and $\Omega$.
To see this, we let $h_\varep\in W^{1,q}(\Omega;\br^3)$ and $Q_\varep \in W^{1,2}(\Omega)$
be the same functions as in the proof of Lemma \ref{lemma-3.2}.
It follows from Theorem \ref{theorem-2.2} that
\begin{equation}\label{remark-3.2.3}
\aligned
\| u_\varep\|_{L^\infty(\Omega)}
&\le \|\nabla Q_\varep\|_{L^\infty(\Omega)} +\| h_\varep\|_{L^\infty(\Omega)}\\
&\le C \left\{ \| F\|_{C^\gamma(\Omega)} +\| h_\varep\|_{C^\gamma(\Omega)} +\|\nabla_{\tan} Q_\varep\|_{C^\gamma(\partial\Omega)} \right\}.\\
&\le C \left\{ \|F\|_{C^\gamma(\Omega)}
+\| h\|_{W^{1, q}(\Omega)} +\|\nabla_{\tan} Q_\varep\|_{C^\gamma(\partial\Omega)}  \right\}\\
&\le C \left\{ \| F\|_{C^\gamma(\Omega)} +\| G\|_{L^q(\Omega)} +\|\nabla \times u_\varep\|_{L^q(\Omega)}
+\|\nabla_{\tan} Q_\varep\|_{C^\gamma(\partial\Omega)} \right\},
\endaligned
\end{equation}
where we have used the Sobolev embedding for the third inequality and (\ref{3.2.5}) for the
fourth. This, together with the estimate
$$
\aligned
\|\nabla_{\tan} Q_\varep\|_{C^\gamma(\partial\Omega)}
 & \le \| n\times \nabla Q_\varep \|_{C^\gamma (\partial\Omega)}\\
& \le \| n\times u_\varep\|_{C^\gamma(\partial\Omega)} + \| n\times h_\varep\|_{C^\gamma(\partial\Omega)}  \\
&  \le \| f\|_{C^\gamma(\partial\Omega)} +C\, \| h_\varep\|_{W^{1,q}(\Omega)}\\
& \le \| f\|_{C^\gamma (\partial\Omega)} + C \left\{ \|\nabla \times u_\varep\|_{L^q(\Omega)}
+\| G\|_{L^q(\Omega)} \right\},
\endaligned
$$
gives (\ref{remark-3.2.1}).
}
\end{remark}

We are now ready to prove Theorem \ref{theorem-A}.

\begin{proof}[\bf Proof of Theorem \ref{theorem-A}]
Given $f\in  L^p(\partial\Omega;\br^3)$ with $n\cdot f=0$ on $\partial\Omega$ and Div$(f)\in W^{-\frac{1}{p}, p} (\partial\Omega)$,
it follows from \cite[Theorem 11.6]{MMP} that there exists $u\in V^p(\Omega)$ such that $n\times u=f$ on $\partial\Omega$ and
$$
\| u\|_{L^p(\Omega)} +\|\nabla\times  u\|_{L^p(\Omega)} \le C_p\left\{ \| f\|_{L^p(\partial\Omega)}
+\|\text{Div} (f)\|_{W^{-\frac{1}{p}, p} (\partial\Omega)} \right\}.
$$
Consequently, by considering $u_\varep-u$ in $\Omega$, we may assume that $f=0$.

We first consider the case $p>2$.
The uniqueness follows from the uniqueness in the case $p=2$.
Let $F\in L^p(\Omega;\br^3)$, $G\in L^p(\Omega;\br^3)$, and $u_\varep$ be the unique solution of
(\ref{max-equation}) in $V_0^2(\Omega)$.
To establish the $L^p$ estimate (\ref{L-p-estimate}), we further assume that $2<p\le 6$.
Since $\|u_\varep\|_{L^2(\Omega)} +\|\nabla \times u_\varep\|_{L^2(\Omega)} \le C \left\{
\| F\|_{L^2(\Omega)} +\| G\|_{L^2(\Omega)} \right\}$,
it follows from (\ref{estimate-3.1}) that
$$
\aligned
\|\nabla \times u_\varep \|_{L^p(\Omega)}
& \le C \left\{ \| F\|_{L^2(\Omega)} +\| G\|_{L^p(\Omega)} +\| u_\varep\|_{L^2(\Omega)} \right\}\\
& \le C\left\{ \| F\|_{L^p(\Omega)} +\| G\|_{L^p(\Omega)}\right\}.
\endaligned
$$
Similarly, by the estimate (\ref{estimate-3.2}),
$$
\aligned
\| u_\varep\|_{L^p(\Omega)}
&\le C \left\{ \| F\|_{L^p(\Omega)} +\| G\|_{L^2(\Omega)} +\| \nabla \times u_\varep\|_{L^2(\Omega)} \right\}\\
&  \le C\left\{ \| F\|_{L^p(\Omega)} +\| G\|_{L^p(\Omega)}\right\}.
\endaligned
$$

Suppose now that $p>6$. Let $\frac{1}{q}=\frac{1}{p}+\frac13$. Then $2<q<3$ and we have proved that
$$
\|u_\varep\|_{L^q(\Omega)} +\|\nabla\times u_\varep\|_{L^q(\Omega)}
\le C \left\{ \| F\|_{L^q(\Omega)} +\| G\|_{L^q(\Omega)} \right\}.
$$
As before, we may use estimates (\ref{estimate-3.1}) and (\ref{estimate-3.2}) to obtain
$$
\aligned
\|\nabla \times u_\varep\|_{L^p(\Omega)}
&\le C \left\{ \| F\|_{L^q(\Omega)} +\|G\|_{L^p(\Omega)} +\| u_\varep\|_{L^q(\Omega)} \right\}\\
& \le C \left\{ \| F\|_{L^p(\Omega)} +\| G\|_{L^p(\Omega)}\right\},
\endaligned
$$
and
$$
\aligned
\| u_\varep\|_{L^p(\Omega)}
&\le C \left\{ \|F\|_{L^p(\Omega)} + \| G\|_{L^q(\Omega)} +\|\nabla \times u_\varep\|_{L^q(\Omega)} \right\}\\
&\le C \left\{ \| F\|_{L^p(\Omega)} +\| G\|_{L^p(\Omega)}\right\}.
\endaligned
$$

Finally, we handle the case $1<p<2$ by a duality argument.
Let $F, G\in C_0^\infty(\Omega; \br^3)$ and $u_\varep$ be the solution in $V_0^2(\Omega)$ of (\ref{max-equation}).
Let $F_1$, $G_1 \in C_0^\infty(\Omega;\br^3)$ and $v_\varep$ be the solution in $V_0^2(\Omega)$ of
$$
\left\{
\aligned
\nabla \times \big( A^*(x/\varep)\nabla \times v_\varep\big) + B^*(x/\varep) v_\varep & = F_1 +\nabla \times G_1 & \quad
& \text{ in } \Omega,\\
n\times v_\varep & =0 & \quad & \text{ on } \partial\Omega,
\endaligned
\right.
$$
where $A^*$ and $B^*$ are the adjoints of $A$ and $B$, respectively.
Since $A^*$ and $B^*$ satisfy the same conditions as $A$ and $B$, we see that
\begin{equation}\label{3.3.1}
\| v_\varep\|_{L^{p^\prime}(\Omega)}
+\|\nabla \times v_\varep\|_{L^{p^\prime}(\Omega)}
\le C \left\{ \|F_1\|_{L^{p^\prime}(\Omega)}
+\|G_1\|_{L^{p^\prime}(\Omega)} \right\}.
\end{equation}
Note that
\begin{equation}\label{3.3.2}
\aligned
& \int_\Omega F_1\cdot u_\varep\, dx +\int_\Omega G_1\cdot \nabla \times u_\varep\, dx\\
&=\int_\Omega A(x/\varep)\nabla \times u_\varep \cdot \nabla\times v_\varep\, dx
+\int_\Omega B(x/\varep) u_\varep \cdot v_\varep\, dx\\
& =\int_\Omega F\cdot v_\varep\, dx
+\int_\Omega G\cdot \nabla \times v_\varep\, dx.
\endaligned
\end{equation}
This, together with (\ref{3.3.1}), yields
\begin{equation}\label{3.3.3}
\| u_\varep\|_{L^p(\Omega)}
+\|\nabla \times u_\varep\|_{L^p(\Omega)}
\le C \left\{ \| F\|_{L^p(\Omega)} +\|G\|_{L^p(\Omega)} \right\},
\end{equation}
by duality. With the estimate (\ref{3.3.3}) at our disposal, the existence of solutions in $V^p(\Omega)$
as well as the estimate (\ref{L-p-estimate}) for
arbitrary data $F, G\in L^p(\Omega;\br^3)$ follows readily by a density argument.
Observe  that the duality relation (\ref{3.3.2}) holds as long as $u_\varep \in V^p(\Omega)$ and $v_\varep\in V^{p^\prime}(\Omega)$
are solutions of (\ref{max-equation}) and its adjoint system, respectively.
The uniqueness for $p<2$ also follows from (\ref{3.3.2}) and (\ref{3.3.1}) by duality.
This completes the proof of Theorem \ref{theorem-A}.
\end{proof}

%%%%%%%%%%%%%%%%%%%%%%%%%%%%%%

%%%%%%%%%%%%%%%%%%%%%%%%%%%%%%

\section{Proof of Theorem \ref{theorem-B}}

Choose $q>3$ such that $\gamma =1-\frac{3}{q}$.
It follows from estimates  (\ref{remark-3.1.1}) and (\ref{remark-3.2.1}) that
$$
\aligned
 \| u_\varep\|_{L^\infty(\Omega)}
& +\|\nabla \times u_\varep\|_{L^\infty(\Omega)}\\
& \le C \big\{ \| F\|_{C^\gamma(\Omega)} +\| G\|_{C^\gamma(\Omega)}
+\| f\|_{C^\gamma (\partial\Omega)}
+\| \text{Div} (f)\|_{C^\gamma(\partial\Omega)}\\
&\qquad\qquad\qquad\qquad\qquad\qquad
+\| u_\varep\|_{L^q(\Omega)} +\| \nabla \times u_\varep\|_{L^q(\Omega)} \big\}.
\endaligned
$$
This, together with the $L^q$ estimate of $u_\varep$ and $\nabla\times u_\varep$
in Theorem \ref{theorem-A}, gives (\ref{L-infty-estimate}).

\bibliography{ss1}

\bigskip

\begin{flushleft}

% \begin{flushleft}
Zhongwei Shen, Department of Mathematics,
University of Kentucky,
Lexington, Kentucky 40506,
USA.

E-mail: zshen2@uky.edu

\end{flushleft}

 \begin{flushleft}
 Liang Song,
Department of Mathematics,
Sun Yat-sen  University,
Guangzhou, 510275,
P. R. China.

E-mail: songl@mail.sysu.edu.cn

\end{flushleft}

\medskip

\noindent \today

\end{document}